\documentclass[runningheads]{llncs}
\usepackage{graphicx}
\usepackage{amsmath} 
\usepackage{amssymb}  
\usepackage{enumerate}
\usepackage{color}

\newcommand{\R}{\ensuremath{\mathbb{R}}}

\newcommand{\D}{\ensuremath{\mathcal{D}}}
\newcommand{\Es}{\ensuremath{\mathbb{S}}}

\begin{document}
\mainmatter              
\title{Variational problems on Riemannian manifolds with constrained accelerations}
\titlerunning{Variational problems with constrained accelerations}  
%
\author{Alexandre Anahory Simoes\inst{1} \and Leonardo Colombo\inst{2}
}
\authorrunning{Alexandre A. Simoes et al.} 
%
\tocauthor{Alexandre Anahory Simoes, and Leonardo Colombo}
\institute{Instituto de Ciencias Matem\'aticas (CSIC-UAM-UCM-UC3M), Calle Nicol\'as Cabrera 13-15, Cantoblanco, 28049, Madrid, Spain \\
\email{alexandre.anahory@icmat.es} \and
Centre for Automation and Robotics
(CSIC-UPM), Ctra. M300 Campo Real, Km 0,200, Arganda
del Rey - 28500 Madrid, Spain. \\
\email{leonardo.colombo@car.upm-csic.es},
}

\maketitle              

\begin{abstract}
We introduce variational problems on Riemannian manifolds with constrained acceleration and derive necessary conditions for normal extremals in the constrained variational problem. The problem consists on minimizing a higher-order energy functional, among a set of admissible curves defined by a constraint on the covariant acceleration. In addition, we use this framework to address the elastic splines problem with obstacle avoidance in the presence of this type of contraints.

\end{abstract}
\section{Introduction}
Higher-order Variational problems on Riemannian manidfolds have been very well studied in the literature in the last three decades (see for instance \cite{BCC2019,CSLC2001,CL95}). This study comes motivated by the theory of geodesics, presented, for instance, in Milnor \cite{Milnor}, showing a deep and rich example of the close relationship between first order variational problems and differential geometry. Motivated by this connection and applications to dynamic interpolation on manifolds \cite{Noa:89}, Crouch and Silva Leite \cite{CL95} started the development of an interesting geometric theory of generalized cubic  polynomials on a Riemannian manifold $M$, in particular on compact connected Lie groups endowed with a left-invariant metric. Further extensions appears in the context of  obstacle avoindance problems \cite{BCC2019,BCNH}, regression problems on Lie groups \cite{kuper}, collision avoidance problems \cite{CollAvoid}, and sub-Riemannian geometry, with connections with non-holonomic mechanics and control, studied by Bloch and Crouch \cite{BlCr,blochcrouch}. These sub-Riemannian problems are determined by additional constraints on a non-integrable distribution on $M$.

In this work we aim to study generalized cubic polynomails subjected to a specific type of higher-order constraints: linear constraints on the covariant accelerations. This variational problem is closely related to the optimal control problems of underactuated mechanical systems which are particular types of affine connection control systems, studied, for instance, in \cite{CA2014}.

Problems with nonholonomic constraints have been previously introduced \cite{BCC2019,CL95} where the authors deduce necessary conditions for the existence of normal as well as abnormal extremals. Moreover, in \cite{BCC2019}, the authors studied the problem of dynamic interpolation for obstacle avoidance under the presence of constraints on the velocities. Our contribution here is also the introduction of a variational collision avoidance with second-order nonholonomic constraints which are closely related to affine connection control systems.

The remainder of the paper is organized as follows. In Section \ref{background}, we review the subject of constraints on velocities in collision avoidance, in Section \ref{lin:acceleration:sec}, we give necessary condition for normal extremals of collision avoidance with constraints on the covariant accelerations. In Section \ref{Obs:avoidance:sec}, we discuss obstacle avoidance in this setting. Examples and simulation results are also presented in Section \ref{Obs:avoidance:sec}.

\section{Variational problems with linear constraints on the velocities}\label{background}

Given a $n$-dimensional Riemannian manifold $Q$ equipped with the metric $\langle\cdot,\cdot \rangle$, we denote by $\Omega$ be the set of all (possibly piecewise) smooth curves $c:[0,T]\rightarrow Q$ satisfying
\begin{equation}\label{boundary:conditions}
	c(0)=q_{0}, \ c(T)=q_{T}, \ \dot{c}(0)=v_{0}, \ \dot{c}(T)=v_{T},
\end{equation}
where $q_{0}, q_{T} \in Q$, $v_{0}\in T_{q_{0}}Q$ and $v_{T}\in T_{q_{T}}Q$. Note that the tangent space to $\Omega$ at a curve $c$ is the set of vector fields $X:[0,T]\rightarrow TQ$ along $c$ satisfying
\begin{equation*}
	X(0)=0, \ X(T)=0, \ \frac{DX}{dt}(0)=0, \ \frac{DX}{dt}(T)=0,
\end{equation*}
where $D/dt$ is the covariant derivative with respect to the Levi-Civita connection.

Suppose that we want to find minimizers of the functional
\begin{equation*}
\text{min} \int_{0}^{T} \frac{1}{2}\left\langle \frac{D^{2} c}{dt^{2}},\frac{D^{2} c}{dt^{2}} \right\rangle dt
\end{equation*}
among curves in $\Omega$ that also satisfy the following linear constraints on the velocities
\begin{equation}\label{nh:constraints}
\left\langle \mu^{a}(c(t)), \dot{c}(t) \right\rangle=0, \quad a=1, ..., n-k,
\end{equation}
where $\{\mu^{a}\}$ are $n-k$ 1-forms on $Q$. 

The span of $\D$ defines the annihilator space of a distribution $\D = \cup \D_{q}$, with fiber defined by
$\D_{q} = \{ v_{q}\in T_{q}Q \ | \ \langle \mu^{a}(q), v_{q} \rangle, \ a=1, ...,n-k \}$. The rank of the distribution is $k=\text{dim}(\D_{q})$. If $(q^{i})$ is a coordinate chart on $Q$, then the linear constraints are locally given by $\mu_{k}^{a}(q)\dot{q}^{k}=0$.

The necessary conditions satisfied by solutions of the preceding variational problem are classically obtained using the constrained Lagrangian approach, where one considers an augmented Lagrangian function and then applies the Lagrange multiplier's theorem to obtain constrained solutions. However, when we have constraints on velocities or higher-order constraints care must be taken because, by following this procedure, we obtain only \textit{normal extremals}, though \textit{abnormal extremals}, which do not satisfy the same necessary conditions, may also exist. In this paper, we will restrict ourselves to the study of normal extremals.

If the augmented Lagrangian $\mathcal{L}:T^{(2)}Q\times \R^{n-k}\rightarrow \R$ is given by
\begin{equation*}
\mathcal{L} = \frac{1}{2}\left\langle \frac{D^{2} c}{dt^{2}},\frac{D^{2} c}{dt^{2}} \right\rangle + \lambda_{a}\left\langle \mu^{a}(c(t)), \dot{c}(t) \right\rangle,
\end{equation*}
the corresponding second-order Euler-Lagrange equations, which are,
\begin{equation*}
\frac{d^{2}}{dt^{2}}\left(\frac{\partial \mathcal{L}}{\partial \ddot{q}}\right) - \frac{d}{dt}\left(\frac{\partial \mathcal{L}}{\partial \dot{q}}\right) + \frac{\partial \mathcal{L}}{\partial q} = 0,
\end{equation*}
might be given, an intrinsic form (i.e., independent of the choice of local coordinates). In order to do so, let $X_{a}\in \mathfrak{X}(Q)$ be the vector field defined by
\begin{equation}\label{metric:equivalent:vf}
\langle \mu^{a}, Y \rangle = \langle X_{a},Y \rangle, \quad \forall Y \in \mathfrak{X}(Q),
\end{equation}
i.e., $\flat(X_{a})=\mu^{a}$, then it is shown in \cite{CL95} that Euler-Lagrange equations are rewritten in the following form
\begin{equation*}
\frac{D^{4}c}{dt^{4}} + R\left(\frac{D^{2}c}{dt^{2}}, \dot{c}\right)\dot{c} = \dot{\lambda}_{a} X_{a} + \lambda_{a} S(\dot{c}),
\end{equation*}
where $S_{a}:TQ\rightarrow TQ$ is the fiberwise linear bundle operator defined by
\begin{equation}\label{second:form}
d\mu^{a}(X,Y)=\langle S_{a}(X),Y \rangle, \quad \forall X,Y \in \mathfrak{X}(Q).
\end{equation}

	

\begin{remark}A generalization of this problem was studied in \cite{BlCaCoCDC,BCC2019}, where the authors add a potential into the picture to study dynamical interpolation problems for obstacle avoidance and further used in \cite{ABBCC2018,CoGo20,CCCBB2020,CollAvoid,existence,goodman} to provide necessary and sufficient conditions of collision avoidance of multi-agent systems. There one studies curves in $\Omega$ satisfying the constraints \eqref{nh:constraints} and minimizing the functional
\begin{equation*}
\int_{0}^{T} \frac{1}{2}\left(\left\langle \frac{D^{2} c}{dt^{2}},\frac{D^{2} c}{dt^{2}} \right\rangle + \sigma \left\|\frac{dc}{dt} \right\|^{2} + V(c(t))\right) dt
\end{equation*}
with $V:Q\rightarrow \R$ a potential function and $\sigma\geqslant 0$. In this case, the necessary conditions for $c$ to be a normal extremal are given by the following equations
\begin{equation*}
\frac{D^{4}c}{dt^{4}} + R\left(\frac{D^{2}c}{dt^{2}}, \dot{c}\right)\dot{c} - \sigma  \frac{D^{2} c}{dt^{2}} + \frac{1}{2}\text{grad} V = \dot{\lambda}_{a} X_{a} + \lambda_{a} S(\dot{c}).
\end{equation*}
\end{remark}

\section{Variational problems with linear constraints on the covariant accelerations}\label{lin:acceleration:sec}

Given a $n$-dimensional Riemannian manifold $Q$ equipped with the metric $\langle\cdot,\cdot \rangle$, we want to find solutions of the minimal acceleration problem
\begin{equation}\label{Problem:2}
\text{min} \int_{0}^{T} \frac{1}{2}\left\langle \frac{D^{2} c}{dt^{2}},\frac{D^{2} c}{dt^{2}} \right\rangle dt
\end{equation}
but with the covariant accelerations are subjected to the following linear constraints
\begin{equation}\label{Cov:accel:constraints}
\left\langle \mu^{a}(c(t)), \frac{D^{2} c}{dt^{2}} \right\rangle=0, \quad a=1, ..., n-k,
\end{equation}
where $\{\mu^{a}\}$ are $n-k$ 1-forms on $Q$. 

The span of $\D$ defines the annihilator space of a distribution $\D = \cup \D_{q}$, with fiber defined by
$\D_{q} = \{ v_{q}\in T_{q}Q \ | \ \langle \mu^{a}(q), v_{q} \rangle, \ a=1, ...,n-k \}$. The rank of the distribution is $k=\text{dim}(\D_{q})$. If $(q^{i})$ is a coordinate chart on $Q$, then the linear constraints are locally given by
\begin{equation*}
\mu_{k}^{a}(q)(\ddot{q}^{k}+\Gamma_{ij}^{k}\dot{q}^{i}\dot{q}^{j})=0,
\end{equation*}
where $(\Gamma_{ij}^{k})$ are the Christoffel symbols of the Levi-Civita connection.

These kind of problems has appeared before in the control literature in a different form. Suppose that an orthonormal reference frame for $\D$ exists and denote it by $\{Y_{1},...,Y_{k}\}$. Then this constrained variational problem is equivalent to the following optimal control problem:
\begin{equation*}
\text{min} \int_{0}^{T} \frac{1}{2}\left\langle u,u \right\rangle dt
\end{equation*}
subjected to the dynamics
$\dot{c}=V, \, \frac{DV}{dt}=u^{i} Y_{i}, \, (u^{1},...,u^{k}) \in \R^{k}$.

Recall that the covariant acceleration along the curve $c$ is locally defined by
\begin{equation*}
\frac{DV}{dt}=\dot{V} + \Gamma_{ij}^{k} V^{i}\dot{c}^{j}\frac{\partial}{\partial q^{k}},
\end{equation*}
This type of control system is known as \textit{affine connection control system} (see, e.g., \cite{L2000,LM97,CA2014}). It is actually a special type of controlled simple mechanical systems with forces (see \cite{BL2004}) which have the following general form:
\begin{equation*}
\dot{c}=V, \quad \frac{DV}{dt}=\sharp(F(c,V,u)),
\end{equation*}
where $\sharp:T^{*}Q \rightarrow TQ$ is the musical isomorphism associated to the Riemannian metric and $F:TQ\times \R^{k}\rightarrow T^{*}Q$ is a force map, possibly depending on the controls. In the case of affine connection control systems, we have no potential or external forces, apart from those associated with the controls.

The relationship between affine connection control systems and the variational problem \eqref{Problem:2} without constraints was already discussed in \cite{BGK2016} for the case where the control system is fully actuated, meaning that $k=n$. However, in our case, the constrained variational problem is rather related with the \textit{underactuated} control system ($k<n$). 
The necessary conditions satisfied by normal extremals of the preceding problem are obtained from the augmented Lagrangian $\mathcal{L}:T^{(2)}Q\times \R^{n-k}\rightarrow \R$ given by
\begin{equation*}
\mathcal{L} = \frac{1}{2}\left\langle \frac{D^{2} c}{dt^{2}},\frac{D^{2} c}{dt^{2}} \right\rangle + \lambda_{a}\left\langle \mu^{a}(c(t)), \frac{D^{2} c}{dt^{2}} \right\rangle.
\end{equation*}

As before, we will write the corresponding Euler-Lagrange equations in a geometric form.

\begin{lemma}\label{useful:lemma}
	The fiberwise linear operator $S_{a}:TQ\rightarrow TQ$ defined by equation \eqref{second:form} satisfies
	\begin{equation}
		\langle S_{a}(X),Y \rangle = \langle \nabla_{X} X_{a}, Y \rangle - \langle \nabla_{Y}X_{a}, X \rangle,
	\end{equation}
	for all $X, Y \in \mathfrak{X}(Q)$.
\end{lemma}

\begin{proof}
	Using the definition of $S_{a}$ in equation \eqref{second:form} and the definition of the differential on a 1-form we have that
	$\langle S_{a}(X),Y \rangle = X\langle X_{A},Y \rangle - Y \langle X_{a},X \rangle - \langle X_{a}, [X,Y] \rangle.$
	Using the fact that the Levi-Civita connection is torsionless and compatible with the metric we obtain
$\langle S_{a}(X),Y \rangle  = \langle \nabla_{X} X_{a}, Y \rangle - \langle \nabla_{Y}X_{a}, X \rangle + \langle X_{a}, \nabla_{X}Y-\nabla_{Y}X - [X,Y] \rangle	 = \langle \nabla_{X} X_{a}, Y \rangle - \langle \nabla_{Y}X_{a}, X \rangle$, as we desired to show.\hfill$\square$
\end{proof}

Before stating the main theorem in this section we recall the following well-known lemma from the literature.

\begin{lemma}\label{lemma:curvature}
	The variation $c:[-\varepsilon,\varepsilon] \times [0,T]\rightarrow Q$ satisfies
	$$\frac{D}{ds}\frac{D^{2}c}{dt^{2}} = \frac{D^{2}}{dt^{2}} \frac{\partial c}{\partial s} + R\left(\frac{\partial c}{\partial s},\frac{\partial c}{\partial t}\right)\frac{\partial c}{\partial t}.$$
\end{lemma}

\begin{theorem}\label{main:theorem}
	If a trajectory $c:I\rightarrow Q$ is a normal extremal of the functional \eqref{Problem:2} satisfying the constraints given by \eqref{Cov:accel:constraints} then $c$ satisfies
	\begin{equation*}
		\begin{split}
			\frac{D^{4}c}{dt^{4}} + R\left(\frac{D^{2}c}{dt^{2}}, \dot{c}\right)\dot{c} & + \ddot{\lambda}_{a}X_{a} + 2\dot{\lambda}_{a}\frac{D X_{a}}{dt} + \lambda_{a}\frac{D^{2} X_{a}}{dt^{2}} \\
			& + \lambda_{a}\left(R(X_{a},\dot{c})\dot{c} + \nabla_{\frac{D^{2}c}{dt^{2}}} X_{a} - S_{a}\left( \frac{D^{2}c}{dt^{2}} \right)\right) = 0,
		\end{split}
	\end{equation*}
	where $R$ is the curvature of the Levi-Civita connection.
\end{theorem}

\begin{proof}
	Suppose $c$ is a minimum trajectory. Then, from Lagrange Multiplier's Theorem we have that $c$ must be a critical value of the functional
	\begin{equation*}
		\mathcal{J}(c, \lambda_{a})=\int_{0}^{T} \left(\frac{1}{2}\left\langle \frac{D^{2} c}{dt^{2}},\frac{D^{2} c}{dt^{2}} \right\rangle + \lambda_{a}\left\langle \mu^{a}(c(t)), \frac{D^{2} c}{dt^{2}} \right\rangle \right) \ dt,
	\end{equation*}
	which means that
$\displaystyle{\frac{d}{ds}\left. \mathcal{J}(c_{s}, \lambda_{a,s}) \right|_{s=0}=0}$ for all variations of the curve $(c,\lambda)$, i.e., for all maps $(c_{s}, \lambda_{a,s})$ satisfying $(c_{0}, \lambda_{a,0})=(c,\lambda)$.
	By computing the critical values of $\mathcal{J}$, we find that the first term gives
	$\displaystyle{\int_{0}^{T} \left\langle \frac{D^{4}c}{dt^{4}} + R\left(\frac{D^{2}c}{dt^{2}}, \dot{c}\right)\dot{c}, \delta c \right\rangle}$, 
	where $\displaystyle{\delta c := \left. \frac{\partial c_{s}}{\partial s} \right|_{s=0}}$	while the second term splits in two sub-terms. 
	
	The first one is
	$\displaystyle{\int_{0}^{T} \delta \lambda_{a} \left\langle \mu^{a}(c(t)), \frac{D^{2} c}{dt^{2}} \right\rangle \ dt}$ with
	$\displaystyle{\delta \lambda := \left. \frac{\partial \lambda_{a,s}}{\partial s} \right|_{s=0}}$	and the second is 
	$\displaystyle{\int_{0}^{T} \lambda_{a} \left.\frac{d}{ds}\right|_{s=0}\left\langle \mu^{a}(c_{s}(t)), \frac{D^{2} c_{s}}{dt^{2}} \right\rangle \ dt}$, and in order to deal with it let us consider again the vector fields $X_{a}\in \mathfrak{X}(Q)$ defined in equation \eqref{metric:equivalent:vf}, i.e., $\flat(X_{a})=\mu^{a}$. Then
	\begin{equation}\label{main:issue}
		\left.\frac{d}{ds}\right|_{s=0}\left\langle X_{a}(c_{s}(t)), \frac{D^{2} c_{s}}{dt^{2}} \right\rangle =  \left\langle \frac{D X_{a}}{ds}(c(t)), \frac{D^{2} c}{dt^{2}} \right\rangle + \left\langle X_{a}(c(t)), \frac{D}{ds}\frac{D^{2} c}{dt^{2}} \right\rangle.
	\end{equation}
	Using Lemma \ref{useful:lemma} with $X=\frac{D^{2}c}{dt^{2}}$ and $Y=\delta c$, we have that the first term in equation \eqref{main:issue} gives
	$$\left\langle \frac{D X_{a}}{ds}(c(t)), \frac{D^{2} c}{dt^{2}} \right\rangle = \left\langle \nabla_{\frac{D^{2}c}{dt^{2}}} X_{a}, \delta c \right\rangle - \left\langle S_{a}\left( \frac{D^{2}c}{dt^{2}} \right), \delta c \right\rangle.$$
	Observe that $\nabla_{\frac{D^{2}c}{dt^{2}}} X_{a}$ is well-defined and its value may be defined using an extension of $\frac{D^{2}c}{dt^{2}}$ to a vector field $Z$ satisfying $Z(c(t))=\frac{D^{2}c}{dt^{2}}$. Then,
	$\nabla_{\frac{D^{2}c}{dt^{2}}} X_{a} = \nabla_{Z} X_{a} |_{c(t)}$.
	
	To simplify the second term in equation \eqref{main:issue} we first use Lemma \ref{lemma:curvature}, obtaining
	$$\left\langle X_{a}(c(t)), \frac{D}{ds}\frac{D^{2} c}{dt^{2}} \right\rangle = \left\langle X_{a}(c(t)), \frac{D^{2}}{dt^{2}}\delta c + R(\delta c, \dot{c})\dot{c} \right\rangle.$$
	Using the symmetries from the curvature tensor the term involving the curvature reduces to
	$\left\langle R(X_{a}, \dot{c})\dot{c},\delta c \right\rangle$ while the other term is simplified using integration by parts. Indeed we get that
	\begin{equation*}
		\begin{split}
			\int_{0}^{T} \lambda_{a} \left\langle X_{a}, \frac{D^{2}}{dt^{2}}\delta c \right\rangle = -\int_{0}^{T} & \left(\dot{\lambda_{a}} \left\langle X_{a}, \frac{D}{dt}\delta c \right\rangle + \lambda_{a}\left\langle \frac{D X_{a}}{dt}, \frac{D}{dt}\delta c \right\rangle \right) dt \\
			& + \left[\lambda_{a}\left\langle X_{a}, \frac{D}{dt}\delta c \right\rangle \right]_{t=0}^{t=T}.
		\end{split}
	\end{equation*}
	Integrating by parts again we get on one hand
	\begin{equation*}
		\begin{split}
			\int_{0}^{T} \dot{\lambda_{a}} \left\langle X_{a}, \frac{D}{dt}\delta c \right\rangle = -\int_{0}^{T} & \left(\ddot{\lambda_{a}} \left\langle X_{a}, \delta c \right\rangle + \dot{\lambda}_{a}\left\langle \frac{D X_{a}}{dt}, \delta c \right\rangle \right) dt \\
			& + \left[\dot{\lambda}_{a}\left\langle X_{a}, \delta c \right\rangle \right]_{t=0}^{t=T},
		\end{split}
	\end{equation*}
	while on the other hand we get
	\begin{equation*}
		\begin{split}
			\int_{0}^{T} \lambda_{a}\left\langle \frac{D X_{a}}{dt}, \frac{D}{dt}\delta c \right\rangle = -\int_{0}^{T} & \left(\dot{\lambda_{a}} \left\langle \frac{D X_{a}}{dt}, \delta c \right\rangle + \lambda_{a}\left\langle \frac{D^{2} X_{a}}{dt^{2}}, \delta c \right\rangle \right) dt \\
			& + \left[\dot{\lambda}_{a}\left\langle \frac{D X_{a}}{dt}, \delta c \right\rangle \right]_{t=0}^{t=T}.
		\end{split}
	\end{equation*}
	Therefore, the curves $(c,\lambda)$ are a critical value of $\mathcal{J}$ if the integral
	\begin{equation*}
		\begin{split}
			\int_{0}^{T} & \left( \left\langle \frac{D^{4}c}{dt^{4}} + R\left(\frac{D^{2}c}{dt^{2}}, \dot{c}\right)\dot{c} + \ddot{\lambda}_{a}X_{a} + 2\dot{\lambda}_{a}\frac{D X_{a}}{dt} + \lambda_{a}\frac{D^{2} X_{a}}{dt^{2}} + R(X_{a},\dot{c})\dot{c}, \delta c \right\rangle \right. \\
			& + \left. \lambda_{a}\left\langle \nabla_{\frac{D^{2}c}{dt^{2}}} X_{a} - S_{a}\left( \frac{D^{2}c}{dt^{2}} \right), \delta c \right\rangle + \delta \lambda_{a} \left\langle X_{a}, \frac{D^{2} c}{dt^{2}} \right\rangle \right) dt
		\end{split}
	\end{equation*}
	vanishes, since all the boundary terms vanish. By applying the fundamental lemma of calculus of variations, we conclude the proof.\hfill$\square$
\end{proof}




\begin{example}
	Suppose we have a planar rigid body which is free to move on every direction of the plane. Its configuration space is $SE(2)\cong \R^{2} \times \Es^{1}$ parametrized by $(x,y,\theta)$. Its Riemannian metric is given by
	$g=m(dx \otimes dx + dy \otimes dy) + J d\theta \otimes d\theta$.
	Hence, the corresponding Levi-Civita connection has vanishing Christoffel symbols. However, this rigid body can only be actuated in some directions. In particular,  we have the following dynamic underactuated control system
	\begin{equation*}
			m \ddot{x}  = \cos \theta  \ u_{r}, \quad
			m \ddot{y}  = \sin \theta  \ u_{r}, \quad
			J \ddot{\theta}  = u_{\theta}.
	\end{equation*}
	
	Consider the optimal control problem consisting of the dynamical equations above and by the cost function given by $\displaystyle{\int_{0}^{T} \frac{1}{2}(u_{r}^{2} + u_{\theta}^{2}) dt}$.
	Notice that the control system implies the following constraints on the covariant accelerations $\displaystyle{
	\left\langle \mu, \frac{D^{2} c}{dt^{2}} \right\rangle = \sin\theta \ddot{x} - \cos \theta \ddot{y} = 0}$ with $\mu = \sin\theta d{x} - \cos \theta d{y}$. Thus, the optimal control problem is equivalent to the constrained variational problem
	\begin{equation*}
		\int_{0}^{T} \frac{1}{2}\left\langle \frac{D^{2} c}{dt^{2}}, \frac{D^{2} c}{dt^{2}} \right\rangle dt \quad \text{and} \quad  \left\langle \mu, \frac{D^{2} c}{dt^{2}} \right\rangle = 0.
	\end{equation*}
	Hence, necessary conditions for normal extremals are given by
	\begin{equation*}
		\begin{split}
				x^{(4)} + \frac{\ddot{\lambda}}{m}\sin\theta + 2  \frac{\dot{\lambda}}{m}\dot{\theta}\cos \theta + \frac{\lambda}{m}(\ddot{\theta}\cos\theta - \dot{\theta}^{2}\sin \theta ) & = 0 \\
				y^{(4)} - \frac{\ddot{\lambda}}{m}\cos\theta + 2  \frac{\dot{\lambda}}{m}\dot{\theta}\sin \theta + \frac{\lambda}{m}(\ddot{\theta}\sin\theta + \dot{\theta}^{2}\cos \theta ) & = 0 \\
				\theta^{(4)} + \frac{\lambda}{J}(\cos \theta \ddot{x} + \sin \theta \ddot{y}) & = 0,
		\end{split}
	\end{equation*}
	where we used the fact that
	$X = \frac{\sin \theta}{m}\frac{\partial}{\partial x}-\frac{\cos \theta}{m}\frac{\partial}{\partial y}$
	and also that
	$\nabla_{\frac{D^{2}c}{dt^{2}}} X = \ddot{\theta}\left( \frac{\cos \theta}{m}\frac{\partial}{\partial x} + \frac{\sin \theta}{m}\frac{\partial}{\partial y} \right)$
	and
	$$S\left(a\frac{\partial}{\partial x}+b\frac{\partial}{\partial y}+c\frac{\partial}{\partial z} \right) = \frac{c}{m} \cos \theta \frac{\partial}{\partial x} + \frac{c}{m} \sin \theta \frac{\partial}{\partial y} - \frac{a \cos \theta + b \sin \theta}{J} \frac{\partial}{\partial \theta}.$$
\end{example}

\section{Application to obstacle avoidance with constraints on the accelerations}\label{Obs:avoidance:sec}

Now consider the problem of finding a curve $c \in \Omega$ satisfying the constraints given by \eqref{Cov:accel:constraints} and, for $V:Q\rightarrow \R$ a potential function minimize the functional
\begin{equation*}
\int_{0}^{T} \frac{1}{2}\left(\left\langle \frac{D^{2} c}{dt^{2}},\frac{D^{2} c}{dt^{2}} \right\rangle + \sigma \left\|\frac{dc}{dt} \right\|^{2} + V(c(t))\right) dt
\end{equation*}

We will now deduce necessary conditions in order for a curve $c:[0,T]\rightarrow Q$ to be an extremal of the constrained variational problem above.

\begin{theorem}
	A necessary condition for a curve $c$ to be a normal extremal of the previous functional subjected to the constraints in \eqref{Cov:accel:constraints} is that it satisfies
	\begin{equation*}
		\begin{split}
			\frac{D^{4}c}{dt^{4}} + R\left(\frac{D^{2}c}{dt^{2}}, \dot{c}\right)\dot{c} & -\sigma \frac{D^{2}c}{dt^{2}} + \ddot{\lambda}_{a}X_{a} + 2\dot{\lambda}_{a}\frac{D X_{a}}{dt} + \lambda_{a}\frac{D^{2} X_{a}}{dt^{2}} \\
			& + \lambda_{a}\left(R(X_{a},\dot{c})\dot{c} + \nabla_{\frac{D^{2}c}{dt^{2}}} X_{a} - S_{a}\left( \frac{D^{2}c}{dt^{2}} \right)\right) + \frac{1}{2}\emph{grad } V = 0,
		\end{split}
	\end{equation*}
	where $V:Q\rightarrow \R$ is a potential function.
\end{theorem}

\begin{proof}
	The result follows directly from the proof of Theorem \ref{main:theorem}, to which we must add the variation of the terms containing the potential function $V$ and the velocities. Using the notation in that proof, on one hand we have that
	$$\frac{1}{2}\int_{0}^{T}\left. \frac{d}{ds} \right|_{s=0} V(c_{s}(t)) \ dt = \frac{1}{2}\int_{0}^{T} dV(c(t))(\delta c) \ dt.$$
	By definition, $\text{grad } V\in \mathfrak{X}(Q)$ is the unique vector field satisfying
	$dV(q)(Z)= \langle \text{grad } V (q), Z \rangle$, $\forall \ Z \in \mathfrak{X}(Q)$, 
	from where the term containing $\text{grad } V$ follows.
	
	On the other hand, we have
	$\displaystyle{\frac{1}{2}\int_{0}^{T}\left. \frac{d}{ds} \right|_{s=0} \sigma \left\|\frac{dc}{dt} \right\|^{2} \ dt = \int_{0}^{T} \sigma \left\langle \dot{c},\frac{D^{2} c}{dsdt} \right\rangle}$. 
	Integrating by parts we get
	$$\frac{1}{2}\int_{0}^{T}\left. \frac{d}{ds} \right|_{s=0} \sigma \left\|\frac{dc}{dt} \right\|^{2} \ dt = - \int_{0}^{T} \sigma \left\langle \frac{D^{2} c}{dt^2},\delta c \right\rangle + \left[\langle \dot{c}, \delta c \rangle \right]_{t=0}^{t=T},$$
	from where the result follows, since the boundary terms vanish.\hfill$\square$
\end{proof}

\begin{example}
	Suppose we have again the planar rigid body from the last example. Suppose there is a circular shaped obstacle with centre located at $(0,0)$ and radius $r$ in the $xy$-plane. Consider the artificial obstacle avoidance potential given by
	$\displaystyle{V(x,y,\theta)=\frac{\tau}{x^{2}+y^{2}-r^{2}}}$, penalizing collision with the obstacle, where $\tau>0$. Normal extremals must satisfy the equations
	\begin{equation*}
		\begin{split}
			x^{(4)} - \sigma \ddot{x} + \frac{\ddot{\lambda}}{m}\sin\theta + 2  \frac{\dot{\lambda}}{m}\dot{\theta}\cos \theta + \frac{\lambda}{m}(\ddot{\theta}\cos\theta - \dot{\theta}^{2}\sin \theta ) - \frac{\tau x}{(x^{2}+y^{2}-r^{2})^{2}}& = 0, \\
			y^{(4)} - \sigma \ddot{y} - \frac{\ddot{\lambda}}{m}\cos\theta + 2  \frac{\dot{\lambda}}{m}\dot{\theta}\sin \theta + \frac{\lambda}{m}(\ddot{\theta}\sin\theta + \dot{\theta}^{2}\cos \theta ) - \frac{\tau y}{(x^{2}+y^{2}-r^{2})^{2}}& = 0, \\
			\theta^{(4)} - \sigma \ddot{\theta} + \frac{\lambda}{J}(\cos \theta \ddot{x} + \sin \theta \ddot{y}) & = 0.
	\end{split}
	\end{equation*}
	
	In Figure \ref{fig:my_label} we show a simulation of our method. A shooting method and a fpurth order Runge-Kutta method with $h=0.1$, $N=55$, $T=Nh$, are used to simulate the boundary value problem. The  curve represents a normal extrema avoiding a static obstacle with $r=0.1$. The parameters for the trajectory used are $m=1$, $J=2$, $\sigma= 0.1$, $\tau = 1$. Boundary condition are given by: $q(0)=(-1,-1,0)$, $q(T)=(1,1,0)$, $v(0)=(0.1,0,0.2)$, $v(T)=(0,0,0)$,
	
	\begin{figure}[htb!]
	    \centering
	    \includegraphics[scale=0.5]{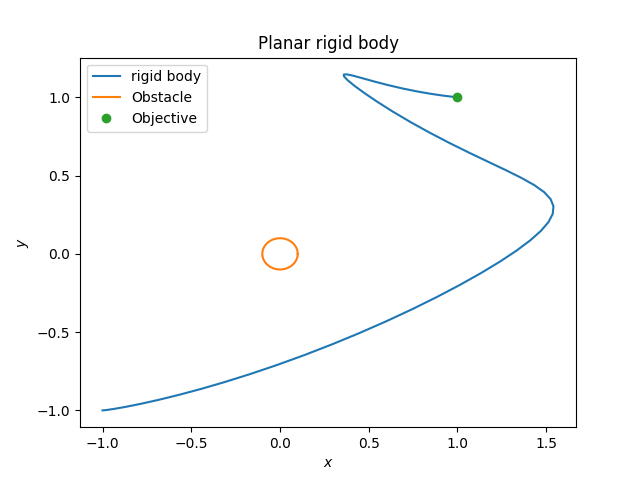}
	    \caption{A normal extrema avoiding a static obstacle.}
	    \label{fig:my_label}
	\end{figure}
\end{example}


\section{Conclusions and future work}

We have studied variational problems on Riemannian manifolds with constrained acceleration, derived necessary conditions for normal extrema, and we have also stablihed the close relation of our problem with affine connection control systems.

Interesting questions may arrise from the study of variational problem associated with the optimal control of affine connection control system. For instance, it is known that a dynamic control nonholonomic system, might be expressed as an affine connection control system using the nonholonomic connection $\nabla^{nh}$ (see \cite{LM97}) giving equations of the form:
$\nabla^{nh}_{\dot{c}}\dot{c} = u^{i}Y_{i}$, with $\{Y_{i}\}$ spanning $\D$. It would be interesting to study the corresponding variational problem.


\section*{Acknowledgements}

The authors acknowledge financial support from the Spanish Ministry of Science and Innovation, under grants PID2019-106715GB-C21, MTM2016-76702-P.
%
%
%

\end{document}